\input ./style/arxiv-vmsta.cfg
\documentclass[numbers,compress,v1.0.1]{vmsta}
\usepackage{vtexbibtags}
\usepackage{enumitem}
\volume{7}
\issue{2}
\pubyear{2020}
\firstpage{203}
\lastpage{219}
\aid{VMSTA154}
\doi{10.15559/20-VMSTA154}
\articletype{research-article}

\DeclareMathOperator{\Exvtex}{\mathbf{E}}
\DeclareMathOperator{\Cov}{\mathbf{Cov}}
\DeclareMathOperator{\Prob}{\mathbf{P}}
\DeclareMathOperator{\Var}{\mathbf{Var}}
\startlocaldefs

\newtheorem{thm}{Theorem}
\newtheorem{remark}{Remark}
\newtheorem{lemma}{Lemma}

\allowdisplaybreaks
\ifdefined\HCode 
\def\index#1{}
\else 
\fi
\endlocaldefs

\begin{document}

\begin{frontmatter}
\pretitle{Research Article}

\title{Prediction in polynomial errors-in-variables models}

\author{\inits{A.}\fnms{Alexander}~\snm{Kukush}\thanksref{cor1}\ead[label=e1]{alexander\_kukush@univ.keiv.ua}}
\author{\inits{I.}\fnms{Ivan}~\snm{Senko}\ead[label=e2]{ivan\_senko@ukr.net}}
\thankstext[type=corresp,id=cor1]{Corresponding author.}
\address{\institution{Taras Shevchenko National University of Kyiv}, Kyiv, \cny{Ukraine}}




\begin{abstract}
A~multivariate errors-in-variables (EIV) model with an~intercept
term, and a~polynomial EIV model are considered.
Focus is made on~a~structural homoskedastic
case, where vectors of~covariates are i.i.d. and measurement errors are
i.i.d. as~well. The~covariates contaminated with~errors are normally distributed
and the~corresponding classical errors are also assumed normal. In~both
models, 
it is shown that (inconsistent) ordinary least squares estimators of~regression
parameters yield an~a.s. approximation to~the~best prediction of~response
given the values of observable covariates. Thus, not~only in~the~linear
EIV, but in~the~polynomial EIV models as~well, consistent estimators of~regression
parameters are useless in~the~prediction problem, provided the size and
covariance structure of observation errors for the predicted subject do not
differ from those in the data used for the model fitting.
\end{abstract}


\received{\sday{24} \smonth{4} \syear{2020}}
\revised{\sday{4} \smonth{5} \syear{2020}}
\accepted{\sday{6} \smonth{5} \syear{2020}}
\publishedonline{\sday{25} \smonth{5} \syear{2020}}

\end{frontmatter}
\section{Introduction}
\label{Intro}

We deal with~errors-in-variables (EIV) models which are widely used in~system
identification~\cite{S18}, epidemiology~\cite{CRSC06}, econometrics~\cite{WM00},
etc. In~such regression models\index{regression ! models} (with~unknown parameter $\beta $), the response
variable\index{response variable} $y$ depends on~the covariates $z$ and $\xi $, where $z$ is observed
precisely and $\xi $ is observed with~error. We consider the~\emph{classical}
measurement error\index{measurement error} $\delta $, i.e., instead of~$\xi $ the~surrogate data
$x = \xi + \delta $ is observed; moreover, the~model is
\emph{structural}, i.e. $z$, $\xi $ and $\delta $ are mutually independent,
and we have i.i.d. copies of the~model ($z_{i}$, $\xi _{i}$,
$\delta _{i}$, $x_{i} = \xi _{i} + \delta _{i}$, $y_{i}$),
$i = 1, \dots , n$. The~measurement error\index{measurement error} can be
\emph{nondifferential}, when the~distribution of~$y$ given
$\left ( \xi , z, x \right )$ depends only on $\left ( \xi , z
\right )$, and \emph{differential}, otherwise \cite[Section~2.5]{CRSC06}.

The present paper is devoted to~the~prediction\index{prediction} of~the~response variable\index{response variable} from
$\xi $ and $z$. Based on the~observations ($y_{i}$, $z_{i}$, $x_{i}$),
$i = 1, \dots , n$, and given new values $z_{0}$ and $x_{0}$ of~$z$ and
$x$ variables, we want to~predict either the~new $y_{0}$ (this procedure
is called \emph{individual prediction}) or the~exact relation
$\eta _{0} = \Exvtex \left [ \left . y_{0} \right | z_{0}, \xi _{0}
\right ]$, where $\xi _{0}$ is a~new value for~$\xi $ (this procedure is
called \emph{mean prediction}). Both~prediction problems are important in~econometrics~\cite{GPG17}.
The individual prediction is used in~the~Leave-one-out cross-validation
procedure.

The~best mean squared error individual predictor is
%
\begin{equation}
\label{bestInd}
\hat{y}_{0} = \Exvtex \left [ \left . y_{0} \right | z_{0}, x_{0} \right ]
\end{equation}
and the~best mean squared error predictor\index{predictor} of~$\eta _{0}$ is
%
\begin{equation}
\label{bestMean}
\hat{\eta }_{0} = \Exvtex \left [ \left . \eta _{0} \right | z_{0}, x_{0}
\right ].
\end{equation}
For~the nondifferential measurement error,
\begin{equation*}
\hat{\eta }_{0} = \Exvtex \left [~\Exvtex \left [ \left . y_{0} \right | z_{0},
\xi _{0}, x_{0} \right ]~\right | z_{0}, x_{0}] = \Exvtex \left [ \left . y_{0}
\right | z_{0}, x_{0} \right ] = \hat{y}_{0},
\end{equation*}
and the~best mean predictor coincides with~the~best individual predictor,
but this needs not to hold for~the~differential measurement error.\index{measurement error}

Both predictors\index{predictor} \eqref{bestInd} and \eqref{bestMean} are
\emph{unfeasible}, because they involve unknown model parameters. Our~goal
is to~construct consistent estimators of~the~predictors\index{predictor} as~the~sample size~$n$
grows.

The nonparametric individual prediction\index{prediction} under errors in covariates is studied
in \cite{MM15}. Below we consider only parametric models.

For~scalar linear EIV models with~normally distributed $\xi $ and
$\delta $, it is stated in~\cite[Section~2.5.1]{CV99} that the~ordinary
least squares (OLS) predictor\index{OLS ! predictor} should be used even when dealing with~the~EIV
model. This is quite surprising, since the~OLS estimator\index{OLS ! estimator} of~$\beta $ is
inconsistent due~to~the~attenuation effect~\cite{CV99}. In fact, there
is no surprise that in a \emph{Gaussian} model the linear OLS estimator\index{OLS ! estimator}
provides a consistent prediction, since the Gaussian dependence is always linear.
In the present paper, we consider a \emph{non-Gaussian} regression model, since
the distribution of the observable covariate $z$ is not assumed Gaussian; therefore,
the consistency of OLS predictions\index{OLS ! predictions} in such a model is a nontrivial
feature.

We confirm the assertion, that the OLS estimator\index{OLS ! estimator} yields a suitable prediction
under the model validity, for~two kinds of~EIV models: multivariate linear\index{multivariate linear EIV model}
and polynomial\index{polynomial EIV model}. For~this~purpose, we just follow the~recommendation of~\cite[Section~2.6]{CV99}
and analyze the~regression\index{regression} of $y$ on~the~observable $z$ and $x$. In~other
nonlinear EIV models,\index{nonlinear EIV models} the~OLS predictor\index{OLS ! predictor} (contaminated from~the~initial
regression\index{regression} $y$ on $\left (z, \xi \right )$, where we naively substitute
$x$ for~$\xi $) is inconsistent; instead the~least-squares predictor can
be used from~the~regression\index{regression} $y$ on $\left (z, x \right )$.

The~paper is organized as~follows. In~Sections~\ref{Lin_EIV} and \ref{Poly_EIV}, we state the~results on~prediction\index{prediction} in~multivariate linear\index{multivariate linear EIV model}
and polynomial EIV models,\index{polynomial EIV model} respectively. Section~\ref{Other_EIV} studies
briefly some other nonlinear EIV models,\index{nonlinear EIV models} and Section~\ref{Conclusion} concludes.

Through the~paper, all vectors are column ones, $\Exvtex $ stands for~the~expectation
and acts as~an~operator on~the~total product, and
$\Cov \left ( x \right )$ denotes the~covariance matrix of~a~random~vector
$x$. By~$I_{p}$ we denote the~identity matrix of~size~$p$. For~symmetric
matrices $A$ and $B$ of~the~same size, $A>B$ and $A \geq B$ means that
$A-B$ is positive definite or positive semidefinite, respectively.

\section{Prediction\index{prediction} in~a~multivariate linear EIV model\index{multivariate linear EIV model}}
\label{Lin_EIV}

\subsection{Model and main assumptions}
\label{sec2.1}

Consider a~multivariate linear EIV model\index{multivariate linear EIV model} with~the intercept term\index{intercept term} (structural
case):
%
\begin{align}
y &= b + C^{T} z + B^{T} \xi + e + \epsilon ,
\label{model_y}
\\
x &= \xi + \delta .
\label{model_x}
\end{align}
Here the random vector $y$ is the~response variable distributed in~$
\mathbb{R}^{d}$; the random vector $z$ is the~observable covariate distributed
in~$\mathbb{R}^{q}$, the random vector $\xi $ is the~unobservable (latent) covariate
distributed in~$\mathbb{R}^{m}$; $x$ is the~surrogate data observed instead
of~$\xi $; $e+\epsilon $ is the~random error in~$y$, $\delta $ is the~measurement
error\index{measurement error} in~the~latent covariate;\index{latent covariate} $C \in \mathbb{R}^{q \times d}$,
$B \in \mathbb{R}^{m \times d}$ and $b \in \mathbb{R}^{d}$ contain unknown
regression parameters, where $b$ is the~intercept term.\index{intercept term} The random vector
$e$ models the error in~the~regression equation,\index{regression ! equation} and $\epsilon $ models the measurement
error\index{measurement error} in~$y$; $\epsilon $ can be correlated with~$\delta $.

Such models are studied, e.g., in \cite{VV91,S18,Sh18} in~relation to~system identification problems and numerical
linear algebra. We list the model assumptions.
\begin{enumerate}[label={\bf(\roman*)},ref=(\roman*),align=left,leftmargin=*]
\item\label{indErrors} Three vectors $z$, $\xi $, $e$ and the augmented
measurement error\index{measurement error} vector
$\left ( \epsilon ^{T}, \delta ^{T} \right )^{T}$ are independent with~finite
2nd moments; the errors $\epsilon $ and $\delta $ can be correlated.
\item\label{nonsingRegr} The covariance matrices
$\Sigma _{z}:= \Cov (z)$ and $\Sigma _{x}:= \Cov (x)$ are nonsingular.
\item\label{zeromeanErrs} The errors $e$, $\epsilon $ and
$\delta $ have zero means.
\item\label{gaussErrors} The errors $\epsilon $, $\delta $ and
covariate $\xi $ are jointly Gaussian.
\end{enumerate}

Introduce the~cross-covariance matrix
\begin{equation*}
\Sigma _{\epsilon \delta }:= \Exvtex \epsilon \delta ^{T}.
\end{equation*}
The~classical measurement error $\delta $\index{measurement error} is nondifferential if, and only
if, $\epsilon $ and $\delta $ are independent, i.e.
$\Sigma _{\epsilon \delta } = 0$ (see Section~\ref{Intro} for~the~definition
of the nondifferential error).

We denote also
\begin{gather}
\mu = \Exvtex x, \qquad \Sigma _{\xi }= \Cov (\xi ), \qquad
\nonumber
\\
\Sigma _{e}= \Cov (e), \qquad \Sigma _{\epsilon }= \Cov (\epsilon ),
\qquad \Sigma _{\delta }= \Cov (\delta ),
\nonumber
\\
\label{defSigma}
\Sigma _{11} = \operatorname{block-diag} (\Sigma _{\xi }, \Sigma _{\epsilon }), \qquad \Sigma _{12} =
\begin{bmatrix}
\Sigma _{\xi
}\\
\Sigma _{\epsilon \delta }%
\end{bmatrix}
, \qquad \Sigma _{22} = \Sigma _{x}.
\end{gather}
Thus, $\Sigma _{11}$ is a~block-diagonal matrix, and sometimes we will
use $\Sigma _{22}$ for~the~covariance matrix of~$x$.

\subsection{Regression\index{regression} of~$y$ on $z$ and $x$}
\label{sec2.2}

\begin{lemma}
\label{repr_y_lemma}
Assume conditions \emph{\ref{indErrors}} to \emph{\ref{gaussErrors}}.
\renewcommand{\labelenumi}{(\alph{enumi})}
\begin{enumerate}
\item The~response variable\index{response variable}~\eqref{model_y} can be represented as
\begingroup
\abovedisplayskip=7.5pt
\belowdisplayskip=7.5pt
\begin{equation}
\label{repr_y}
y = b_{x} + C^{T} z + B^{T}_{x} x + u,
\end{equation}
where $z$, $x$, and $u$ are independent, $C$ remains unchanged compared
with \eqref{model_y}, $\Exvtex u = 0$, $\Exvtex \|u\|^{2} < \infty $, and
%
\begin{gather}
\label{repr_bx}
b_{x} = b + B^{T} \Sigma _{\delta }\Sigma _{x}^{-1} \mu - \Sigma _{
\epsilon \delta }\Sigma _{x}^{-1} \mu ,
\\
\label{repr_Bx}
B^{T}_{x} = B^{T} \Sigma _{\xi }\Sigma _{x}^{-1} + \Sigma _{\epsilon
\delta }\Sigma _{x}^{-1}.
\end{gather}
\item Assume additionally the~following condition:
\begin{enumerate}[label={\bf(\roman*)},ref=(\roman*),align=left,leftmargin=*,start=5]
\item\label{nonsingU_assump} Either $\Sigma _{e}$ or
$\Sigma _{11}- \Sigma _{12}\Sigma _{22}^{-1} \Sigma _{12}^{T}$ is positive
definite.
\end{enumerate}

Then the~error term $u$ in \eqref{repr_y} has a positive definite covariance
matrix, $\Sigma _{u}$.
\end{enumerate}
\endgroup
\end{lemma}

\begin{proof}
(a) Introduce the jointly Gaussian vectors
\begingroup
\abovedisplayskip=7.5pt
\belowdisplayskip=7.5pt
\begin{equation*}
x^{(1)} =
\begin{pmatrix}
\xi
\\
\epsilon
\end{pmatrix}
, \qquad x^{(2)} = x.
\end{equation*}
We have
\begin{gather*}
\mu ^{(1)} := \Exvtex x^{(1)} =
\begin{pmatrix}
\mu
\\
0
\end{pmatrix}
, \qquad \mu ^{(2)} := \Exvtex x^{(2)}=\mu ;
\\
\Cov \left (x^{(1)} \right ) = \Sigma _{11}, \qquad \Cov \left ( x^{(2)}
\right ) = \Sigma _{22},
\end{gather*}
which is positive definite by assumption~\ref{nonsingRegr},
\begin{equation*}
\Exvtex \left [ x^{(1)} \left ( x^{(2)} \right )^{T} \right ] = \Sigma _{12},
\end{equation*}
where the~matrices $\Sigma _{11}$, $\Sigma _{12}$, $\Sigma _{22}$ are given
in~\eqref{defSigma}. Now, according to Theorem 2.5.1 \cite{A58} the conditional
distribution of $x^{(1)}$ given $x^{(2)}$ is
\begin{gather}
\left [ \left . x^{(1)} \right | x^{(2)} \right ] \sim \mathcal{N}
\left (\mu _{1|2}, V_{1|2} \right ),
\nonumber
\\
\mu _{1|2} = \mu _{1|2} ( x^{(2)} ) = \mu ^{(1)} + \Sigma _{12}
\Sigma _{22}^{-1} \left ( x^{(2)} - \mu ^{(2)} \right ) =
\begin{pmatrix}
\Sigma _{\delta }\Sigma _{x}^{-1} \mu + \Sigma _{\xi }\Sigma _{x}^{-1} x
\\
\Sigma _{\epsilon \delta }\Sigma _{x}^{-1} (x - \mu )
\end{pmatrix}
,
\nonumber
\\
\label{V12}
V_{1|2} = \Sigma _{11}- \Sigma _{12}\Sigma _{22}^{-1} \Sigma _{12}^{T}.
\end{gather}

Hence
$ (\xi ^{T}, \epsilon ^{T})^{T} - \mu _{1|2}(x) =: (\gamma _{1}^{T},
\gamma _{2}^{T})^{T} $ is uncorrelated with~$x$ and has the Gaussian distribution
$\mathcal{N} \left ( 0, V_{1|2} \right )$. Therefore,
%
\begin{align}
\label{xi_gamma}
\xi &= \Sigma _{\delta }\Sigma _{x}^{-1} \mu + \Sigma _{\xi }\Sigma _{x}^{-1}
x + \gamma _{1},
\\
\label{eps_gamma}
\epsilon &= \Sigma _{\epsilon \delta }\Sigma _{x}^{-1} \left ( x -
\mu \right ) + \gamma _{2}.
\end{align}
\endgroup
Substitute \eqref{xi_gamma} and \eqref{eps_gamma} into~\eqref{model_y}
and obtain the~desired relations~\xch{\eqref{repr_y}--\eqref{repr_Bx}}{\eqref{repr_y}~-- \eqref{repr_Bx}} with
\begin{equation*}
u = e + B^{T} \gamma _{1} + \gamma _{2}.
\end{equation*}
Here $(z, e, x)$ and a~couple $( \gamma _{1}, \gamma _{2})$ are independent,
hence $(z, x, u)$ are independent as well. This implies the~statement (a).

(b) We have
%
\begin{equation}
\label{Cov_u}
\Cov (u) = \Sigma _{e}+ \Cov \left (B^{T} \gamma _{1} + \gamma _{2}
\right ) =: \Sigma _{u}.
\end{equation}
If $\Sigma _{e}> 0$ then $\Sigma _{u}\geq \Sigma _{e}> 0$, thus,
$\Sigma _{u}> 0$; and if $V_{1|2} > 0$ then
$\Sigma _{u}\geq \Cov \left (B^{T} \gamma _{1} + \gamma _{2} \right ) >
0$, thus, $\Sigma _{u}> 0$. This accomplishes the~proof of Lemma~\ref{repr_y_lemma}.
\end{proof}

As~a~particular case take a~model with a univariate response and univariate
regressor~$\xi $.

\begin{lemma}
\label{repr_y_Univ_lemma}
Consider the~model \eqref{model_y}, \eqref{model_x} with~$d = m = 1$. Assume
conditions \emph{\ref{indErrors}}, \emph{\ref{zeromeanErrs}}, and \emph{\ref{gaussErrors}}. Suppose also that
%
\begin{equation}
\label{nonsingUniv}
\Sigma _{z}>0, \qquad \Sigma _{\epsilon }> 0, \qquad \Sigma _{\delta }>0,
\qquad \left |\operatorname{Corr}(\epsilon , \delta ) \right | < 1.
\end{equation}

Then expressions \xch{\eqref{repr_y}--\eqref{repr_Bx}}{\eqref{repr_y} -- \eqref{repr_Bx}} hold true, where the~error
term $u$ has a positive variance $\sigma ^{2}_{u} = \Sigma _{u}$.
\end{lemma}

\begin{proof}
First suppose that $\Sigma _{\xi }>0$. According to~Lemma~\ref{repr_y_lemma},
it is enough to~check that $V_{1|2}$ given in~\eqref{V12} is positive definite.

A direct computation shows that
\begin{equation*}
V_{1|2} = \frac{1}{\sigma ^{2}_{x}}
\begin{pmatrix}
\sigma ^{2}_{\xi }\sigma ^{2}_{\delta }& - \sigma ^{2}_{\xi }\sigma _{
\epsilon \delta }
\\
- \sigma ^{2}_{\xi }\sigma _{\epsilon \delta } & \sigma ^{2}_{\epsilon
}\sigma ^{2}_{x} - \sigma ^{2}_{\epsilon \delta }
\end{pmatrix}
=: \frac{V}{\sigma ^{2}_{x}}.
\end{equation*}
Here in~the~scalar case we write $\sigma ^{2}_{\xi }= \Sigma _{\xi }$,
$\sigma ^{2}_{\delta }= \Sigma _{\delta }$,
$\sigma _{\epsilon \delta } = \Sigma _{\epsilon \delta }$, etc. The matrix~$V$
is positive definite, because
$\sigma ^{2}_{\xi }\sigma ^{2}_{\delta }> 0$ and
\begin{equation*}
\operatorname{det} V = \sigma ^{2}_{\delta }\sigma ^{2}_{x} \left (
\sigma ^{2}_{\epsilon }\sigma ^{2}_{\delta }- \sigma ^{2}_{\epsilon
\delta } \right ) > 0
\end{equation*}
due to~condition~\eqref{nonsingUniv}.

Now, suppose that $\Sigma _{\xi }=0$. Then $\xi =\mu $ almost surely. With
some computations, it can be shown that
$u=e+\epsilon -\sigma _{\epsilon \delta }\sigma _{\delta }^{-2}\delta $ almost
surely, whence
$\sigma _{u}^{2}=\sigma _{e}^{2}+\sigma _{\epsilon }^{2}-\sigma _{
\epsilon \delta }^{2}\sigma _{\delta }^{-2}>0$. Lemma~\ref{repr_y_Univ_lemma}
is proved.
\end{proof}

\subsection{Individual prediction}
\label{sec2.3}

Now, consider independent copies of~the~multivariate model \eqref{model_y}, \eqref{model_x}:
\begin{equation*}
\left ( y_{i}, z_{i}, \xi _{i}, e_{i}, \epsilon _{i}, x_{i}, \delta _{i}
\right ),\quad  i = 1, \dots , n.
\end{equation*}
Based on~the observations
%
\begin{equation}
\label{observations}
\left (y_{i}, z_{i}, x_{i} \right ),\quad i = 1, \dots , n,
\end{equation}
and for~given $z_{0}$, $x_{0}$, we want to~estimate the~individual predictor
$\hat{y}_{0}$ presented in~\eqref{bestInd} and the~mean predictor
$\hat{\eta }_{0}$ presented in~\eqref{bestMean}.

Assume conditions \ref{indErrors} to~\ref{gaussErrors} and suppose
that all model parameters are unknown. Lemma~\ref{repr_y_lemma} implies
the~expansion~\eqref{repr_y} with~$\Exvtex u = 0$. All the~underlying random
vectors have finite 2nd moments, hence
%
\begin{equation}
\label{bestLinPred}
\hat{y}_{0} := b_{x} + C^{T} z_{0} + B^{T}_{x} x_{0}
\end{equation}
is the~best mean squared error predictor\index{predictor} of~$y_{0}$. Since it is unfeasible,
we have to~estimate the~coefficients $b_{x}$, $C$ and $B_{x}$ using the~sample~\eqref{observations}.
The~OLS estimator\index{OLS ! estimator} $\left (\hat{b}_{x}, \hat{C}, \hat{B}_{x} \right )$ minimizes
the~penalty function
\begin{equation*}
Q(b, C, B) := \sum ^{n}_{i=1}{\left \|  y_{i} - b - C^{T} z_{i} - B^{T} x_{i}
\right \|  ^{2}}, \quad b \in \mathbb{R}, \ C \in \mathbb{R}^{q
\times d}, \ B \in \mathbb{R}^{m \times d}.
\end{equation*}

Let bar denote the average over $i = 1, \dots , n$, e.g.,
\begin{equation*}
\bar{y} = \frac{1}{n}\sum ^{n}_{i=1}{y_{i}},
\end{equation*}
and $S_{uv}$ denote the~sample covariance matrix of $u$ and $v$ variables,
e.g.,
%
\begin{equation}
\label{Sample_Cov}
S_{xy} = \frac{1}{n}\sum ^{n}_{i=1}{\left ( x_{i} - \bar{x} \right )
\left (y_{i} - \bar{y} \right )^{T}}, \qquad S_{xx} = \frac{1}{n}\sum ^{n}_{i=1}{
\left ( x_{i} - \bar{x} \right ) \left (x_{i} - \bar{x} \right )^{T}},
\end{equation}
etc. The~OLS estimator\index{OLS ! estimator} can be computed from the~relations~\cite{VV91}
%
\begin{gather}
\label{OLS}
\bar{y} = \hat{b}_{x} + \hat{C}^{T} \bar{z} + \hat{B}^{T}_{x} \bar{x},
\\
\label{OLS_repr}
\begin{pmatrix}
\hat{C}
\\
\hat{B}_{x}
\end{pmatrix}
= S^{+}_{rr} S_{ry}, \qquad r :=
\begin{pmatrix}
z
\\
x
\end{pmatrix}
.
\end{gather}
Hereafter $A^{+}$ is the~pseudo-inverse of a~square matrix $A$; see the
properties of $A^{+}$ in \cite{Sl03}. The corresponding OLS predictor\index{OLS ! predictor} is
%
\begin{equation}
\label{indivPred}
\tilde{y}_{0} := \hat{b}_{x} + \hat{C}^{T} z_{0} + \hat{B}^{T}_{x} x_{0} .
\end{equation}

\begin{thm}
\label{OLSpred_thm}
Assume conditions \emph{\ref{indErrors}} to \emph{\ref{gaussErrors}}. Then
$\tilde{y}_{0}$ presented in~\eqref{indivPred} is a strongly consistent
estimator of~the~best predictor $\hat{y}_{0}$, i.e.
$\tilde{y}_{0} \to \hat{y}_{0} \text{ a.s. as }n$ tends to infinity. Moreover,
%
\begin{equation}
\label{OLSpred_cons}
\forall \tau > 0, \quad  \Prob (\left . \left \|  \tilde{y}_{0} - \hat{y}_{0}
\| > \tau \right | z_{0}, x_{0} \right ) \to 0\quad  \text{a.s. as }n\to
\infty .
\end{equation}
\end{thm}

\begin{proof}
By~Strong Law of~Large Numbers we have a.s. as~$n \to \infty $:
\begin{gather*}
S_{rr} \to \operatorname{block-diag} \left ( \Sigma _{z}, \Sigma _{x}
\right ) > 0,
\\
S_{ry} \to
\begin{pmatrix}
\Cov \left (z, y \right )
\\
\Cov \left (x, y \right )
\end{pmatrix}
=
\begin{pmatrix}
\Sigma _{z}\cdot C
\\
\Sigma _{x}\cdot B_{x}
\end{pmatrix}
,
\\
\begin{pmatrix}
\hat{C}
\\
\hat{B}_{x}
\end{pmatrix}
\to
\begin{pmatrix}
\Sigma _{z}^{-1} \Sigma _{z}\cdot C
\\
\Sigma _{x}^{-1} \Sigma _{x}\cdot B_{x}
\end{pmatrix}
=
\begin{pmatrix}
C
\\
B_{x}
\end{pmatrix}
.
\end{gather*}
This convergence, relation~\eqref{OLS} and the~a.s. convergence of~the~sample
means imply that $\hat{b}_{x} \to b_{x} \text{ a.s.}$ Now, both statements
of Theorem~\ref{OLSpred_thm} follow from~\eqref{indivPred} and \eqref{bestLinPred}.
\end{proof}

It is interesting to~construct an asymptotic confidence region for~the~response
$y_{0}$ based on~the~OLS predictor.\index{OLS ! predictor} Assume \ref{indErrors} to~\ref{gaussErrors}.
It holds
\begin{equation*}
\Cov \left ( \left . y_{0} - \hat{y}_{0} \right | z_{0}, x_{0} \right )
= \Cov (u_{0}) = \Sigma _{u},
\end{equation*}
see \eqref{Cov_u}. Introduce the~estimator
\begin{equation*}
\hat{\Sigma }_{u} = \frac{1}{n} \sum ^{n}_{i=1} \left ( y_{i} - \hat{b}_{x}
- \hat{C}^{T}z_{i}-\hat{B}_{x}^{T} x_{i} \right ) \left ( y_{i} -
\hat{b}_{x} - \hat{C}^{T}z_{i}-\hat{B}_{x}^{T} x_{i} \right )^{T}.
\end{equation*}

\begin{thm}
\label{strong_cons_thm}
Suppose that conditions \emph{\ref{indErrors}} to \emph{\ref{gaussErrors}} hold.
Fix the~confidence pro\-bability $1-\alpha $.
\renewcommand{\labelenumi}{(\alph{enumi})}
\begin{enumerate}
\item Assume additionally \emph{\ref{nonsingU_assump}} and define
%
\begin{equation}
\label{reg1}
E_{\alpha }= \left \{  h \in \mathbb{R}^{d}: \left \|  \left (
\hat{\Sigma }_{u}^{+} \right )^{1\!/2} \left (h - \tilde{y}_{0} \right )
\right \|  ^{2} \leq \frac{d}{\alpha } \right \}  .
\end{equation}
Then
%
\begin{equation}
\label{reg1_assert}
\liminf _{n \to \infty } \Prob \left ( \left . y_{0} \in E_{\alpha
}\right | z_{0}, x_{0} \right ) \geq 1 - \alpha .
\end{equation}
\item Let the~model \xch{\eqref{model_y}--\eqref{model_x}}{\eqref{model_y} -- \eqref{model_x}} be purely normal,
i.e. $z$ is normally distributed and $e = 0$. Assume additionally that
the~matrix~\eqref{V12} is nonsingular. Define
%
\begin{equation}
\label{reg2}
D_{\alpha }= \left \{  h \in \mathbb{R}^{d}: \left \|  \left (
\hat{\Sigma }_{u}^{+} \right )^{1\!/2} \left (h - \tilde{y}_{0} \right )
\right \|  ^{2} \leq \chi ^{2}_{d\alpha } \right \}  ,
\end{equation}
where $\chi ^{2}_{d\alpha }$ is an upper $\alpha $-quantile of~$\chi ^{2}_{d}$
distribution, i.e.
$\Prob \left ( \chi ^{2}_{d} > \chi ^{2}_{d\alpha } \right ) =
\alpha $. Then
%
\begin{equation}
\label{reg2_assert}
\lim _{n \to \infty } \Prob \left ( \left . y_{0} \in D_{\alpha
}\right | z_{0}, x_{0} \right ) = 1 - \alpha .
\end{equation}
\end{enumerate}
\end{thm}

\begin{proof}
If $b_{x}$, $C$, and $B_{x}$ were known, then we could approximate
$\Sigma _{u}$ as~follows:
%
\begin{gather}
\label{Sigma_u_approx}
\frac{1}{n} \sum ^{n}_{i=1} u_{i} u^{T}_{i} \to \Sigma _{u}\quad
\text{a.s. as }n \to \infty ,
\\
u_{i} := y_{i} - b_{x} - C^{T} z_{i} - B_{x}^{T} x_{i}.
\nonumber
\end{gather}
Since $u_{i} u^{T}_{i}$ is a~quadratic function of~the~coefficients
$b_{x}$, $C$, $B_{x}$, and the~OLS estimators\index{OLS ! estimator} of~those coefficients are
strongly consistent, the~convergence~\eqref{Sigma_u_approx} remains valid
if we replace all $u_{i}$ with~the~residuals
\begin{equation*}
\hat{u}_{i} := y_{i} - \hat{b}_{x} - C^{T} z_{i} - \hat{B}^{T}_{x} x_{i}.
\end{equation*}
Hence
%
\begin{equation}
\label{Sigma_u_conv}
\hat{\Sigma }_{u}\to \Sigma _{u}\quad \text{a.s. as }n \to \infty .
\end{equation}

(a) Under \ref{nonsingU_assump}, $\Sigma _{u}$ is nonsingular by~Lemma~\ref{repr_y_lemma}(b).
It holds
\begin{equation*}
\Prob \left ( \left .\left \|  \Sigma _{u}^{-1\!/2} \left ( y_{0} -
\hat{y}_{0}\xch{\right )}{) \right )} \right \|  ^{2} > \frac{d}{\alpha } \right | z_{0},
x_{0} \right ) \leq \alpha
\frac{\Exvtex \left \|  \Sigma _{u}^{-1\!/2} \cdot u \right \|  ^{2}}{d} =
\alpha .
\end{equation*}
Since the relations \eqref{OLSpred_cons} and \eqref{Sigma_u_conv} hold true,
the~relations \eqref{reg1_assert}, \eqref{reg1} follow.

(b) Again, in~this purely normal model the~matrix $\Sigma _{u}$ is nonsingular;
conditional on~$z_{0}$ and $x_{0}$, the difference
$y_{0} - \hat{y}_{0} = u_{0}$ has the normal distribution
$\mathcal{N} \left ( 0, \Sigma _{u}\right )$. Then
\begin{equation*}
\Prob \left ( \left . \left \|  \Sigma _{u}^{-1\!/2} \left (y_{0} -
\hat{y}_{0} \right ) \right \|  ^{2} > \chi ^{2}_{d\alpha } \right | z_{0},
x_{0} \right ) = \alpha .
\end{equation*}
Since the relations \eqref{Sigma_u_conv} and \eqref{OLSpred_cons} hold true,
the~relations \eqref{reg2_assert}, \eqref{reg2} follow.
\end{proof}

\begin{remark}
For~the~univariate model with $d \,{=}\, m \,{=}\, 1$, assume the conditions of Lemma~\ref{repr_y_Univ_lemma}.
Then relations \eqref{reg1_assert}, \eqref{reg1} hold true. If additionally
$z = 0$ and $e = 0$ then relations \eqref{reg2_assert} and \eqref{reg2} are valid.
\end{remark}

\subsection{Mean prediction}
\label{sec2.4}

Still consider the~model~\eqref{model_y}, \eqref{model_x} under conditions \ref{indErrors} to \ref{gaussErrors}. We want to~estimate the~mean
predictor $\hat{\eta }_{0}$ presented in~\eqref{bestMean}. We have
\begin{gather*}
\hat{\eta }_{0} = \hat{y}_{0} - \Exvtex \left [ \left . e_{0} \right | z_{0},
x_{0} \right ] - \Exvtex \left [ \left . \epsilon _{0} \right | z_{0}, x_{0}
\right ],
\\
\Exvtex \left [ \left . e_{0} \right | z_{0}, x_{0} \right ] = \Exvtex e_{0} =
0,
\end{gather*}
and by~\eqref{eps_gamma},
\begin{equation*}
\Exvtex \left [ \left . \epsilon _{0} \right | z_{0}, x_{0} \right ] =
\Exvtex \left [ \left . \epsilon _{0} \right | x_{0} \right ] = \Sigma _{
\epsilon \delta }\Sigma _{x}^{-1} \left (x_{0} - \mu \right ).
\end{equation*}
Thus,
%
\begin{equation}
\label{mean_est}
\hat{\eta }_{0} = \hat{y}_{0} - \Sigma _{\epsilon \delta }\Sigma _{x}^{-1}
\left (x_{0} - \mu \right ).
\end{equation}

Based on~observations~\eqref{observations}, strongly consistent and unbiased
estimators of~$\mu $ and $\Sigma _{x}$ are as follows:
%
\begin{gather}
\label{est_mu}
\hat{\mu }= \bar{x} = \frac{1}{n} \sum ^{n}_{i=1}x_{i},
\\
\label{est_Sigma_x}
\hat{\Sigma }_{x}= \frac{1}{n-1} \sum ^{n}_{i=1} \left (x_{i} - \bar{x}
\right ) \left (x_{i} - \bar{x} \right )^{T}.
\end{gather}

\begin{thm}
\label{Meanpred_thm}
Assume conditions \emph{\ref{indErrors}} to~\emph{\ref{gaussErrors}} and suppose
that $\Sigma _{\epsilon \delta }$ is the~only model parameter which is
known. Consider the~estimators~\eqref{indivPred}, \eqref{est_mu}, and \eqref{est_Sigma_x}. Then
\begin{equation*}
\tilde{\eta }_{0} := \tilde{y}_{0} - \Sigma _{\epsilon \delta }
\hat{\Sigma }_{x}^{-1} \left (x_{0} - \hat{\mu }\right )
\end{equation*}
is a strongly consistent estimator of~the~mean predictor~\eqref{bestMean},
and moreover
\begin{equation*}
\forall \tau > 0,\quad  \Prob \left ( \left . \left \|  \tilde{\eta }_{0} -
\hat{\eta }_{0} \right \|  > \tau \right | z_{0}, x_{0} \right ) \to 0\quad
\text{a.s. as }n \to \infty .
\end{equation*}
\end{thm}

\begin{proof}
The~statement follows from relation~\eqref{mean_est}, Theorem~\ref{strong_cons_thm},
and the~strong consistency of~the~estimators~$\hat{\mu }$ and
$\hat{\Sigma _{x}}$.
\end{proof}

Notice that more model parameters should be known in~order to~construct
a~confidence region for~$\eta _{0}$ around $\tilde{\eta }_{0}$.

\section{Prediction\index{prediction} in a polynomial EIV model\index{polynomial EIV model}}
\label{Poly_EIV}

\subsection{Model and main assumptions}
\label{sec3.1}

For~a fixed and known $k \geq 2$, consider a~polynomial EIV model\index{polynomial EIV model} (structural
case):
%
\begin{gather}
\label{model_y_Pol}
y = c^{T} z + \beta _{0} + \beta ^{T} \left (\xi , \xi ^{2}, \dots ,
\xi ^{k} \right )^{T} + e + \epsilon ,
\\
\label{model_x_Pol}
x = \xi + \delta .
\end{gather}
Here the random variable (r.v.) $y$ is the~response variable;\index{response variable} the random vector $z$ is the~observable
covariate distributed in~$\mathbb{R}^{q}$, r.v. $\xi $ is the~unobservable covariate; $x$ is the~surrogate data observed instead
of~$\xi $; $e$ is the~random error in~the~equation, $\epsilon $ and
$\delta $ are the measurement errors\index{measurement error} in~the~response and in~the~latent covariate;\index{latent covariate}
$c \in \mathbb{R}^{q}$, $\beta _{0} \in \mathbb{R}$ and
$\beta = \left ( \beta _{1}, \dots , \beta _{k} \right )^{T} \in
\mathbb{R}^{k}$ contain unknown regression parameters; $\epsilon $ and
$\delta $ can be correlated.

Such models are studied, e.g., in~\cite{CS98,KT16} and applied,
for instance, in econo\-metrics. Let us introduce the model assumptions.
\begin{enumerate}[label={\bf(\alph*)},ref=(\alph*),align=left,leftmargin=*]
\item\label{indErrors_Pol} The random variables $\xi $,
$e$ and random vectors $z$, $\left ( \epsilon , \delta \right )^{T}$ are
independent, with finite 2nd moments; the random variables $\epsilon $ and
$\delta $ can be correlated.
\item\label{nonsingRegr_Pol} The covariance matrix
$\Sigma _{z}:= \Cov (z)$ is nonsingular, and
$\sigma ^{2}_{x} := \Var (x) > 0$.
\item\label{zeromeanErrs_Pol} The errors $e$, $\epsilon $ and
$\delta $ have zero mean.
\item\label{gaussErrors_Pol} The errors $\epsilon $,
$\delta $ and $\xi $ are jointly Gaussian.
\end{enumerate}

We see that assumptions \ref{indErrors_Pol} to~\ref{gaussErrors_Pol}
are similar to~conditions~\ref{indErrors} to~\ref{gaussErrors} imposed
on~the~multivariate linear model, but now the~response and latent covariate\index{latent covariate}
are real valued.

\subsection{Regression $y$\index{regression} on~$z$ and $x$}
\label{sec3.2}

Let us denote
%
\begin{equation}
\label{def_sigma}
\sigma _{\epsilon \delta }= \Exvtex \epsilon \delta , \qquad \mu = \Exvtex x,
\qquad \sigma ^{2}_{\xi }= \Var (\xi ), \qquad \sigma ^{2}_{e}= \Var (e),
\qquad \sigma ^{2}_{\delta }= \Var (\delta ).
\end{equation}

\begin{lemma}
\label{repr_y_Pol_lemma}
Assume conditions \emph{\ref{indErrors_Pol}} to~\emph{\ref{gaussErrors_Pol}}. Then
the~response variable\index{response variable}~\eqref{model_y_Pol} admits the~representation
%
\begin{equation}
\label{repr_y_Pol}
y = c^{T} z + \beta _{0x} + \beta ^{T}_{x} \left (x, x^{2}, \dots , x^{k}
\right )^{T} + u,
\end{equation}
where $z$ and $(x, u)^{T}$ are independent, the vector $c$ remains unchanged
compared\break  with~\eqref{model_y_Pol},
$\Exvtex \left [ \left . u \right | x \right ] = 0$,
$\Exvtex \left [ \left . u^{2} \right | x \right ] < \infty $, and
$\beta _{0x} \in \mathbb{R}$, $\beta _{x} \in \mathbb{R}^{k}$ are transformed
(nonrandom) parameters of~the polynomial regression.\index{polynomial regression}
\end{lemma}

\begin{proof}
In~the new notation, we have from~\eqref{xi_gamma} and \eqref{eps_gamma}:
%
\begin{gather}
\label{xi_gamma_Pol}
\xi = \sigma ^{2}_{\delta }\sigma ^{-2}_{x} \mu + \sigma ^{2}_{\xi
}\sigma ^{-2}_{x} x + \gamma _{1} =: a + K x + \gamma _{1},
\\
\label{eps_gamma_Pol}
\epsilon = \sigma _{\epsilon \delta } \sigma ^{-2}_{x} (x - \mu ) +
\gamma _{2} =: b + f x + \gamma _{2},
\end{gather}
where $z$, $x$ and $(\gamma _{1}, \gamma _{2})^{T}$ are independent, and
$(\gamma _{1}, \gamma _{2})^{T}$ has the Gaussian distribution
$\mathcal{N} \left ( 0, V_{1|2} \right )$\xch{.}{,}

Now, substitute \eqref{xi_gamma_Pol} and \eqref{eps_gamma_Pol} into~\eqref{model_y_Pol}
and get
\begin{gather}
y = c^{T} z + \beta _{0} + \sum ^{k}_{j=1} \beta _{j} (a + K x +
\gamma _{1})^{j} + b + f x + e + \gamma _{2},
\nonumber
\\
\label{y_gamma}
y = c^{T} z + \beta _{0} + \sum ^{k}_{j=1} \beta _{j}\sum ^{j}_{p=0}
\binom{j}{p} (a + K x)^{j-p} \Exvtex \left [ \left . \gamma ^{p}_{1} \right |
x \right ] + b + f x + u,
\\
\label{u_gamma}
u = e + \gamma _{2} + \sum ^{k}_{j=1}\beta _{j} \sum ^{j}_{p=1}
\binom{j}{p} (a + K x)^{j-p} \left ( \gamma ^{p}_{1} - \Exvtex \left [
\left . \gamma ^{p}_{1} \right | x \right ] \right ).
\end{gather}
It holds $\Exvtex \left [ \left . u \right | x \right ]=0$,
$\Exvtex \left [ \left . u^{2} \right | x \right ] < \infty $, and
relations~\xch{\eqref{y_gamma}--\eqref{u_gamma}}{\eqref{y_gamma} -- \eqref{u_gamma}} imply the~statement.
\end{proof}

\subsection{Individual and mean prediction}
\label{sec3.3}

We consider independent copies of~the~polynomial
model~\xch{\eqref{model_y_Pol}--\eqref{model_x_Pol}}{\eqref{model_y_Pol} -- \eqref{model_x_Pol}}:
\begin{equation*}
\left ( y_{i}, z_{i}, \xi _{i}, e_{i}, \epsilon _{i}, x_{i}, \delta _{i}
\right ),\quad i = 1, \dots , n.
\end{equation*}
Based on~observations~\eqref{observations} and for~given $z_{0}$,
$x_{0}$, we want to~estimate the~individual predictor $\hat{y}_{0}$ and
the~mean predictor $\hat{\eta}_{0}$ for~the~polynomial model.

Assume conditions~\ref{indErrors_Pol} to~\ref{gaussErrors_Pol} and
suppose that all model parameters are unknown. Lemma~\ref{repr_y_Pol_lemma}
implies the~expansion~\eqref{repr_y_Pol} with~$\Exvtex \left [u | x, z \right ]=0$. All the~underlying r.v.'s
and the random vector $z$ have finite 2nd moments, hence
\begin{equation*}
\hat{y}_{0} := c^{T} z_{0} + \beta _{0x} + \beta ^{T}_{x} \left ( x_{0},
x^{2}_{0}, \dots , x^{k}_{0} \right )^{T}
\end{equation*}
is the~best mean squared error predictor of~$y_{0}$. We estimate the~coefficients
$c$, $\beta _{0x}$ and $\beta _{x}$ using the~sample~\eqref{observations}
from~the~polynomial model. The~OLS estimator minimizes the~penalty function
\begin{gather*}
Q \left ( c, \beta _{0}, \beta \right ) := \sum ^{n}_{i=1} \left ( y_{i}
- c^{T} z_{i} - \beta _{0} - \beta ^{T} \left ( x_{i}, x^{2}_{i},
\dots , x^{k}_{i} \right )^{T} \right )^{2},
\end{gather*}
$c \in \mathbb{R}^{q}$, $\beta _{0} \in \mathbb{R}$, $\beta \in
\mathbb{R}^{k}$. The~OLS estimator can be computed by~relations similar
to~\xch{\eqref{OLS}--\eqref{OLS_repr}}{\eqref{OLS} -- \eqref{OLS_repr}}:
\begin{gather}
\bar{y} = \hat{c}^{T} \bar{z} + \hat{\beta }_{0x} + \hat{\beta }^{T}_{x}
\left ( \overline{x}, \overline{x^{2}}, \dots , \overline{x^{k}}
\right )^{T},
\nonumber
\\
\label{OLS_repr_Pol}
\begin{pmatrix}
\hat{c}
\\
\hat{\beta }_{x}
\end{pmatrix}
= S^{+}_{rr} S_{ry}, \quad r := (z^{T}, x,\ldots , x^{k})^{T} ;
\end{gather}
the~sample covariance matrices $S_{rr}$ and $S_{ry}$ are defined in~\eqref{Sample_Cov}.
The~corresponding OLS predictor is
%
\begin{equation}
\label{indivPred_Pol}
\tilde{y}_{0} := \hat{c}^{T} z_{0} + \hat{\beta }_{0x} + \hat{\beta }^{T}_{x}
\left ( x_{0}, x^{2}_{0}, \dots , x^{k}_{0} \right )^{T}.
\end{equation}

\begin{thm}
\label{strong_cons_Pol_thm}
Assume conditions \emph{\ref{indErrors_Pol}} to~\emph{\ref{gaussErrors_Pol}}. Then
$\tilde{y}_{0}$ presented in~\eqref{indivPred_Pol} is a~strongly consistent
estimator of~the~individual predictor $\hat{y}_{0}$. Moreover,
\begin{equation*}
\forall \tau > 0, \; \Prob \left ( \left . \left \|  \tilde{y}_{0} -
\hat{y}_{0} \right \|  > \tau \right | z_{0}, x_{0} \right ) \to 0\quad
\text{a.s. as }n \to \infty .
\end{equation*}
\end{thm}

\begin{proof}
Following the~lines of~the~proof of~Theorem~\ref{strong_cons_thm}, it is
enough to~check the~strong consistency of~the~estimators $\hat{c}$ and
$\hat{\beta }_{x}$. We have a.s. as $n \to \infty $:
%
\begin{gather}
\label{S_conv1}
S_{rr} \to \Cov (r) = \operatorname{block-diag} \left ( \Sigma _{z}, D
\right ), \quad D := \Cov (x, x^{2},\ldots x^{k}),
\\
\label{S_conv2}
S_{ry} \to \Cov (r,y) =
\begin{pmatrix}
\Sigma _{z}\cdot c
\\
D \beta _{x}
\end{pmatrix}
.
\end{gather}
By~conditions \ref{nonsingRegr_Pol} and \ref{gaussErrors_Pol},
$x$ is a~nondegenerate Gaussian r.v., therefore, r.v.'s\break
$1,x,\dots ,x^{k}$ are linearly independent in~the Hilbert space
$L_{2} \left ( \Omega , \Prob \right )$ of~square integrable r.v.'s, and
the~covariance matrix $D$ is nonsingular. Relations \eqref{OLS_repr_Pol}, \eqref{S_conv1}, and \eqref{S_conv2} imply that a.s.
as $n \to \infty$
\begin{equation*}
\begin{pmatrix}
\hat{c}
\\
\hat{\beta }_{x}
\end{pmatrix}
\to
\begin{pmatrix}
\Sigma _{z}^{-1} \Sigma _{z}c
\\
D^{-1} D \beta _{x}
\end{pmatrix}
=
\begin{pmatrix}
c
\\
\beta _{x}
\end{pmatrix}
.
\end{equation*}
And the statements of~Theorem~\ref{strong_cons_Pol_thm} follow.
\end{proof}

Similarly to~Theorem~\ref{Meanpred_thm} one can construct a~consistent
estimator of~the~mean predictor~\eqref{bestMean} in~the~polynomial EIV
model. The~strongly consistent estimator of~$\mu $ is given in~\eqref{est_mu}
and the~one~of~$\sigma ^{2}_{x}$ is constructed similarly to~\eqref{est_Sigma_x}:
%
\begin{equation}
\label{est_sigma_x}
\hat{\sigma }^{2}_{x} = \frac{1}{n-1} \sum ^{n}_{i=1} \left (x_{i} -
\bar{x} \right )^{2}.
\end{equation}

\begin{thm}
Assume conditions \emph{\ref{indErrors_Pol}} to~\emph{\ref{gaussErrors_Pol}} and
suppose that $\sigma _{\epsilon \delta }$ defined in~\eqref{def_sigma}
is the~only parameter which is known in~the~model~\eqref{model_y_Pol}, \eqref{model_x_Pol}. Consider the estimators \eqref{indivPred_Pol}, \eqref{est_mu}, and \eqref{est_sigma_x}. Then
\begin{equation*}
\tilde{\eta }_{0} := \tilde{y}_{0} - \sigma _{\epsilon \delta }
\hat{\sigma }_{x}^{-2} ( x_{0} - \hat{\mu })
\end{equation*}
is a~strongly consistent estimator of~the~mean predictor \eqref{bestMean}, and moreover,
\begin{equation*}
\forall \tau > 0, \quad  \Prob \left ( \left . \left \|  \tilde{\eta }_{0} -
\hat{\eta }_{0} \right \|  > \tau \right | z_{0}, x_{0} \right ) \to 0\quad
\text{a.s. as }n \to \infty .
\end{equation*}
\end{thm}

\subsection{Confidence interval for response in quadratic model}
\label{Conf_int_section}

Consider a~quadratic EIV model
%
\begin{gather}
\label{model_y_quad}
y = \beta _{0} + \beta _{1} \xi + \beta _{2} \xi ^{2} + e,
\\
\label{model_x_quad}
x = \xi + \delta .
\end{gather}
It is a~particular case of~the~model~\eqref{model_y_Pol}, \eqref{model_x_Pol} with~$k=2$, $z = 0$ and $\epsilon = 0$.

We use notations~\eqref{def_sigma}. Our~conditions are similar
to~\xch{\ref{indErrors_Pol}--\ref{gaussErrors_Pol}}{\eqref{indErrors_Pol} -- \eqref{gaussErrors_Pol}},
but we assume additionally that the~reliability ratio
%
\begin{equation}
\label{ratio_K}
K := \frac{\sigma ^{2}_{\xi }}{\sigma ^{2}_{x}}
\end{equation}
is separated away from~zero. Thus, assume the~following conditions.
\begin{enumerate}[label={\bf(\alph*)},ref=(\alph*),align=left,leftmargin=*,start=5]
\item\label{quad_var} The random variables $\xi $, $e$ and
$\delta $ are independent; $\xi $ and $\delta $ are Gaussian; $e$ and
$\delta $ have zero mean and $\sigma ^{2}_{e}< \infty $;
$\sigma ^{2}_{x}> 0$.
\item\label{quad_rel} Model parameters are unknown, but
a~lower bound $K_{0}$ for~the~reliability ratio~\eqref{ratio_K} is given,
with~$0 < K_{0} \leq 1\!/2$.
\end{enumerate}

Consider indepedent copies of~the~quadratic model
\begin{equation*}
\left (y_{i}, \xi _{i}, e_{i}, x_{i}, \delta _{i} \right ), \quad i = 0,
1, \dots , n.
\end{equation*}
Based on~observations $(y_{i}, x_{i}), \; i=1, \dots , n$, and for~a given
$x_{0}$, we can construct the~OLS predictor $\tilde{y}_{0}$, see \eqref{indivPred_Pol}, for~$y_{0}$ with~$k=2$, $z_{0} = 0$. Now, we show
the~way how to~construct an~asymptotic confidence interval for~$y_{0}$.
(In~a~similar way this can be done for~a~polynomial EIV model of~higher
order.)

First we write down the~representation~\eqref{y_gamma}, \eqref{u_gamma}. Denote
%
\begin{equation}
\label{m_x}
m_{x} = \Exvtex \left ( \left . \xi \right | x \right ) = K x + \left (1 -
K \right ) \mu .
\end{equation}
We have with independent $m_{x}$ and $\gamma $:
%
\begin{equation}
\label{xi_m_x}
\xi = m_{x} + \gamma , \quad \gamma \sim \mathcal{N} \left (0, K
\sigma ^{2}_{\delta }\right ).
\end{equation}
Then
%
\begin{gather}\label{repr_y_quad}
y = \beta _{0} + \beta _{1} ( m_{x} + \gamma ) + \beta _{2} (m_{x} +
\gamma )^{2} + e=\qquad\qquad\qquad\qquad\qquad\qquad\qquad\quad
\nonumber \\
\qquad\qquad\qquad\   = \beta _{0} + \beta _{1} m_{x} + \beta _{2} \left ( m^{2}_{x} + K
\sigma ^{2}_{\delta }\right ) + u =: \hat{y} + u,
\\
\label{repr_u_quad}
u = e + (\beta _{1} + 2 m_{x} \beta _{2}) \gamma + \beta _{2} \left (
\gamma ^{2} - K \sigma ^{2}_{\delta }\right ).
\end{gather}
Here $\Exvtex \left ( \left . u \right | x \right ) = 0$. From~\eqref{repr_y_quad}
we get that the~best prediction is
\begin{gather}
\hat{y} = \beta _{0x} + \beta _{1x} \cdot x + \beta _{2x} \cdot x^{2},
\nonumber
\\
\label{quad_repr}
\beta _{1x} = \beta _{1} K + 2 \beta _{2} K (1 - K) \mu ,\
\beta _{2x} = \beta _{2} \cdot K^{2}.
\end{gather}
Those coefficients can be estimated using the~strongly consistent OLS estimator,
cf.~\eqref{OLS_repr_Pol},
\begin{equation*}
\begin{pmatrix}
\hat{\beta }_{1x}
\\
\hat{\beta }_{2x}
\end{pmatrix}
= S^{+}_{rr} S_{ry}, \quad r := (x, x^{2})^{T}.
\end{equation*}
The~OLS estimator $\hat{\beta }_{0x}$ satisfies
\begin{equation*}
\bar{y} = \hat{\beta }_{0x} + \hat{\beta }_{1x} \bar{x} + \hat{\beta }_{2x}
\overline{x^{2}},
\end{equation*}
and the~OLS predictor of~$y_{0}$ is equal to
%
\begin{equation}
\tilde{y}_{0} = \hat{\beta }_{0x} + \hat{\beta }_{1x} x_{0} +
\hat{\beta }_{2x} x^{2}_{0}.
\end{equation}

To~construct a~confidence interval for~$y_{0}$, we have to~bound the~conditional
variance of~$u$ given $x_{0}$. From~\eqref{repr_u_quad} we have
\begin{equation*}
\Var \left ( \left . u \right | x \right ) = \sigma _{\epsilon }^{2} +
( \beta _{1} + 2 m_{x} \beta _{2} )^{2} K \sigma ^{2}_{\delta }+ \beta ^{2}_{2}
\cdot 2 \left (K \sigma ^{2}_{\delta }\right )^{2},
\end{equation*}
where
$2 \left ( K \sigma ^{2}_{\delta }\right )^{2} = \Var ( \gamma ^{2} )$. Denote
\begin{equation*}
m_{u^{2}} = \Exvtex \left [\Var \left ( \left . u \right | x \right )
\right ].
\end{equation*}
It holds $\text{ a.s. as }n \to \infty $:
\begin{equation*}
\frac{1}{n} \sum ^{n}_{i=1} \left ( y_{i} - \beta _{0x} - \beta _{1x} x_{i}
- \beta _{2x} x^{2}_{i} \right )^{2} \to m_{u^{2}}.
\end{equation*}
Therefore, we have $\text{ a.s. as }n \to \infty $:
\begin{equation*}
\hat{m}_{u^{2}} := \frac{1}{n} \sum ^{n}_{i=1} \left ( y_{i} -
\hat{\beta }_{0x} - \hat{\beta }_{1x} x_{i} - \hat{\beta }_{2x} x^{2}_{i}
\right )^{2} \to m_{u^{2}}.
\end{equation*}
We have to~bound the~difference
%
\begin{gather}
\label{var_diff}
\Var \left ( \left . u \right | x \right ) - m_{u^{2}} = 4 K \sigma ^{2}_{\delta }\left ( \beta ^{2}_{2} K^{2} \cdot F (K, x, \mu ) + \beta _{1}
\beta _{2} K (x - \mu ) \right ),
\\
F (k, x, \mu ) := x^{2} - \mu ^{2} - \sigma ^{2}_{x}+ 2 K (1 - K)
\mu (x - \mu ).
\nonumber
\end{gather}
Here we used the relations
\begin{gather*}
m_{x} - \Exvtex m_{x} = K (x - \mu ),
\\
m^{2}_{x} - \Exvtex m^{2}_{x} = K^{2} - \mu ^{2} - \sigma ^{2}_{x}+ 2K(1-K)
\mu (x-\mu ).
\end{gather*}
Next, we express \eqref{var_diff} through $\beta _{ix}$ rather than
$\beta _{i}$. Using \eqref{quad_repr} we get:
\begin{gather*}
\sigma ^{2}_{\delta }= \sigma ^{2}_{x}(1 - K),
\\
\begin{aligned}
&\Var \left ( \left . u \right | x \right ) - m_{u^{2}} = 4(1-K)
\sigma ^{2}_{x}\cdot \frac{\beta ^{2}_{2x}}{K} \left ( F (K, x , \mu )
- \frac{2(1-K)}{K}\mu (x-\mu )\right )+
\nonumber \\
&\qquad + 4(1-K)\sigma ^{2}_{x}\beta _{1x}\beta _{2x} \cdot \frac{x-\mu }{K}
\leq 4 \left ( \frac{1}{K_{0}} - 1 \right ) \sigma ^{2}_{x}\cdot G (x,
\mu , \sigma ^{2}_{x}, \beta _{1x}, \beta _{2x}),\nonumber
\end{aligned}
\\
\begin{aligned}
&G (x, \mu , \sigma ^{2}_{x}, \beta _{1x}, \beta _{2x}) = \beta _{2x}^{2}
\biggl [ x^{2} - \mu ^{2} - \sigma ^{2}_{x}+
\nonumber \\
&\qquad\qquad\qquad + 2 \left ( \mu (x-\mu ) \right )_{-} (1 - K_{0})^{2} \left (1 +
\frac{1}{K_{0}} \right ) \biggr ] + \left (\beta _{1x} \beta _{2x} (x -
\mu ) \right )_{+}.
\end{aligned}
\end{gather*}
Here $A_{+} := \max (A, 0)$, $A_{-} := - \min (A, 0)$,
$A \in \mathbb{R}$. Finally,
%
\begin{equation}
\label{Var_quad}
\Var \left ( \left . u \right | x \right ) \leq m_{u^{2}} + 4 \left (
\frac{1}{K_{0}} - 1 \right )\sigma ^{2}_{x}G (x, \mu , \sigma ^{2}_{x},
\beta _{1x}, \beta _{2x}).
\end{equation}

We are ready to~construct a confidence interval for~$y_{0}$.

\begin{thm}
For~the~model~\xch{\eqref{model_y_quad}--\eqref{model_x_quad}}{\eqref{model_y_quad} -- \eqref{model_x_quad}},
assume conditions \emph{\ref{quad_var}} and \emph{\ref{quad_rel}}. Fix the~confidence probability $1-\alpha $. Define
\begin{align*}
I_{\alpha }&= \Biggl \{ h \in \mathbb{R}: \left | h - \tilde{y}_{0}
\right | \leq \alpha ^{-1\!/2} \times
\nonumber \\
&\qquad\qquad\qquad\quad\ {}\times \left [ \hat{m}_{u^{2}} + 4 \left ( \frac{1}{K_{0}} - 1 \right )
\hat{\sigma }^{2}_{x}G (x_{0}, \hat{\mu }, \hat{\sigma }^{2}_{x},
\hat{\beta }_{1x}, \hat{\beta }_{2x}) \right ]^{1\!/2}_{+} \Biggr \},
\end{align*}
where $\hat{m}_{u^{2}}$, $\hat{\sigma }^{2}_{x}$, $\hat{\mu }$,
$\hat{\beta }_{1x}$ and $\hat{\beta }_{2x}$ are strongly consistent estimators
of~the corresponding parameters; the estimators were presented above. Then
\begin{equation*}
\liminf _{n \to \infty } \Prob \left ( \left . y_{0} \in I_{\alpha
}\right | x_{0} \right ) \geq 1-\alpha .
\end{equation*}
\end{thm}

\begin{proof}
It holds for~$t > 0$:
\begin{equation*}
\Prob \left ( \left . \left | y_{0} - \hat{y}_{0} \right | > t \right |
x_{0} \right ) \leq
\frac{\Var \left ( \left . u \right | x_{0} \right )}{t^{2}} \leq
\alpha
\end{equation*}
if
$t$ is selected such that 
$t \geq \alpha ^{-1\!/2}\left [\Var \left ( \left . u \right | x_{0}
\right ) \right ]^{1\!/2}$. Now, the~statement follows from~the~inequality~\eqref{Var_quad}
and the~consistency of~$\tilde{y}_{0}$, $\hat{m}_{u^{2}}$,
$\hat{\sigma }^{2}_{x}$, $\hat{\mu }$, $\hat{\beta }_{1x}$ and
$\hat{\beta }_{2x}$.
\end{proof}

\section{Prediction in~other EIV models}
\label{Other_EIV}

The~OLS predictor $\tilde{y}$ approximates
the~best mean squared error predictor $\hat{y}$ presented in~\eqref{bestInd}
not only in~the~plynomial EIV model.
Let us consider the~model with~exponential regression function
%
\begin{equation}
\label{model_exp}
y = \beta e^{\lambda \xi } + e,\qquad x = \xi + \delta ,
\end{equation}
where the real numbers $\beta $ and $\lambda $ are unknown regression parameters,
and assume condition~\ref{quad_var} from Section~\ref{Conf_int_section}.
Using expansion~\xch{\eqref{xi_m_x}--\eqref{m_x}}{\eqref{xi_m_x} -- \eqref{m_x}}, we get
%
\begin{gather}
y = \beta _{x} \exp \left (\lambda _{x} \cdot x \right ) + u =:
\hat{y} + u,
\\
\beta _{x} = \beta e^{\lambda (1-K) \mu } \cdot \Exvtex e^{\lambda
\gamma }, \ \lambda _{x} = K \lambda , \ \Exvtex e^{\lambda
\gamma } = \exp \left (\frac{\lambda ^{2} K \sigma ^{2}_{\delta }}{2}
\right ),
\nonumber
\\
u=\beta _{x} e^{\lambda _{x}\cdot x}(e^{\lambda \gamma }-\Exvtex e^{
\lambda \gamma }).
\end{gather}

Under~mild conditions, the~OLS predictor
$\tilde{y}_{0} := \hat{\beta }_{x} \exp \left ( \hat{\lambda }_{x}
\cdot x_{0} \right )$ is a~strong\-ly consistent estimator of~$\hat{y}_{0}$,
where $\hat{\beta }_{x}$ and $\hat{\lambda }_{x}$ are the~OLS estimators
of the~regression parameters in~the~model~\eqref{model_exp}.

Similar conclusion can be made for~the~trigonometric model
\begin{equation*}
y = a_{0} + \sum ^{m}_{k=1} \left (a_{k} \cos k \omega \xi + b_{k}
\sin k \omega \xi \right ) + e, \qquad x = \xi + \delta ,
\end{equation*}
where $a_{k}, 0 \leq k \leq m$, $b_{k}, 1\leq k \leq m$, and
$\omega > 0$ are unknown regression parameters.

Finally, we give an~example of~the~model, where the~OLS predictor
\emph{does not} approximate the~best mean squared error predictor. Let
%
\begin{equation}
\label{model_abs}
y = \beta | \xi + a | + e,\qquad x = \xi + \delta ,
\end{equation}
where the real numbers $\beta $ and $a$ are unknown regression parameters,
and assume condition~\ref{quad_var} from~Section~\ref{Conf_int_section};
suppose also that $\sigma ^{2}_{\xi }$ and $\sigma ^{2}_{\delta }$ are positive.

For~$\gamma _{0} \sim \mathcal{N} (0,1)$, evaluate
\begin{equation*}
F(a) := \Exvtex |\gamma _{0} + a| = 2 \phi (a) + a (2 \Phi (a) - 1),
\quad a \in \mathbb{R},
\end{equation*}
where $\phi $ and $\Phi $ are the~pdf and cdf of~$\gamma _{0}$. Then the~best
mean squared error predictor is as follows:
\begin{gather*}
\begin{aligned}
&\hat{y}= \Exvtex \left ( \left . y \right | x \right ) = \beta \Exvtex
\biggl [ \Bigl | a + K x + (1-K) \mu + \sigma _{\delta }\sqrt{K} \gamma _{0}
\Bigr | \biggm | x \biggr ] =
\nonumber \\
&\qquad\qquad\qquad\qquad\qquad = \beta _{x} F \left ( k_{x} \cdot x + b_{x} \right ), \quad k_{x} > 0,
\ \beta _{x} \in \mathbb{R}, \ b_{x} \in \mathbb{R},
\end{aligned}
\\
\beta _{x} = \beta \sigma _{\delta }\sqrt{K}, \ k_{x} =
\frac{\sqrt{K}}{\sigma _{\delta }}, \ b_{x} =
\frac{a + (1 - K) \mu }{\sigma _{\delta }\sqrt{K}}.
\end{gather*}
The~LS estimators $ \hat{k}_{x}$, $\hat{\beta }_{x}$ and $\hat{b}_{x}$ of~$k_{x}$,
$\beta _{x}$ and $b_{x}$ minimize the~penalty function
\begin{equation*}
Q_{LS} (k, \beta , b ) := \sum ^{n}_{i=1} \left (y_{i} - \beta F ( k x_{i}
+ b) \right )^{2}.
\end{equation*}
Under~mild additional conditions, the~LS estimators are strongly consistent,
and the LS~predictor
\begin{equation*}
\tilde{y}_{0} := \beta _{x} F \left ( \hat{b}_{x} \cdot x_{0} + \hat{b}_{x}
\right )
\end{equation*}
converges a.s. to~$\hat{y}_{0} = \Exvtex \left ( \left . y_{0} \right | x_{0} \right ) =
\beta _{x} F \left ( b_{x} \cdot x_{0} + b_{x} \right )$ as the~sample
size grows. Notice that for~this model~\eqref{model_abs}, the~OLS predictor
$\hat{\beta }\left | x_{0} + \hat{a} \right |$ needs not to converge in~probability
to~$\hat{y}_{0}$, where the~OLS estimators $\hat{\beta }$ and $\hat{a}$ minimize
the~penalty function
\begin{equation*}
Q_{OLS} (\beta , a) := \sum ^{n}_{i=1} \left ( y_{i} - \beta | x_{i} +
a | \right )^{2}.
\end{equation*}

\section{Conclusion}
\label{Conclusion}

We considered structural EIV models with~the classical measurement error. We
gave a~list of~models where the~OLS predictor of~response~$y_{0}$ converges
with~pro\-bability one to~the~best mean squared error predictor
$\hat{y}_{0} = \Exvtex \left [ \left . y_{0} \right | z_{0}, x_{0} \right ]$.
In~such models, a functional dependence
$\hat{y}_{0} = \hat{y}_{0} ( z_{0}, x_{0})$ belongs to~the~same parametric
family as the~initial regression function
$\eta _{0}(z_{0}, \xi _{0}) = \Exvtex \left [ \left . y_{0} \right | z_{0},
\xi _{0} \right ]$. Such a~situation looks exceptional for nonlinear models,
and we gave an~example of~model~\eqref{model_abs}, where the~OLS predictor
does not perform well.

We dealt with~both the mean and individual prediction. They coincide in~the~ca\-se
of~nondifferential errors, where it is known that the~errors in~response
and in~covariates are uncorrelated. Otherwise, to~construct the~mean prediction,
one has to~know the~covariance of~the~errors.

In linear models, we managed to~construct an~asymptotic confidence region
for~response around the~OLS prediction, under~totally unknown model parameters.
In the quadratic model, we did it under~the known lower bound of~the~reliability
ratio. The~procedure can be expanded to~polynomial models of~higher order.

Notice that in~linear models without~intercept and in~incomplete polynomial
models (like, e.g. $y = \beta _{0} + \beta _{2} \xi ^{2} + e$,
$x = \xi + \delta $), a prediction with $(z,x)$ naively substituted for~$(z,
\xi )$ in~the~regression of~$y$ on~$(z, x)$ can have huge prediction errors.
As stated in~\cite[Section~2.6]{CRSC06}, predicting $y$ from~$(z, x)$
is merely a~matter of~substituting known values of~$x$ and $z$ into~the~regression
model for~$y$ on~$(z, x)$. We can add that, in~nonlinear EIV models, the~corresponding
error $v = y - \Exvtex \left [ \left . y \right | z,x \right ]$ has the~variance
depending on~$x$, i.e., the~regression of~$y$ on~$(z, x)$ is heteroskedastic;
this should be taken into~account in order to~construct
 a~confidence region for~$y$
in~a~proper way.

Finally, we make a~caveat for~practitioners. \emph{Consistent} EIV regression
parameter estimators are useful especially for~prediction\index{prediction} if the~observation
errors for~the predicted subject differ from~those in~the~data used for~the~model
fitting. This is usually the~case when the~model is fitted by~some experimental
data while the~prediction\index{prediction} is made for~a~real world subject. The~idea to~use
\emph{inconsistent} OLS estimators\index{OLS ! estimator} for~prediction in~this case is not~good.



\begin{acknowledgement}[title={Acknowledgments}]
We are grateful to~Dr.~S.~Shklyar (Kyiv) for~fruitful discussions.
\end{acknowledgement}


\end{document}